\documentclass[12pt]{article}
\usepackage{amsmath,amsthm,amssymb,amsfonts}
\usepackage[active]{srcltx}
\usepackage{latexsym}
\usepackage{enumerate}
\usepackage{verbatim}
\usepackage{multicol}
\usepackage{graphicx}
\usepackage{epsfig}
\usepackage{color}
\usepackage{bezier}
\usepackage{curves}
\usepackage[T1]{fontenc}
\usepackage[ansinew]{inputenc}
\usepackage{fullpage}

\newtheorem{thm}{Theorem}
\newtheorem{lem}{Lemma}
\newtheorem{cor}{Corollary}

\newtheorem{exm}{Example}

\newcommand{\Z}{\mathbb Z}

\author{Thomas Bier\\
\small Dept of Mathematics \& Statistics\\
\small Sultan Qaboos University\\
\small P.O.Box 36, Al-Khoud 123. \\
\small Muscat, Sultanate of Oman\\
{\tt \small tbier@squ.edu.om}
\small \and Imed Zaguia\\
\small Dept of Mathematics \& Computer Science\\
\small Royal Military College of Canada\\
\small P.O.Box 17000, Station Forces \\
\small K7K 7B4 Kingston, Ontario, Canada \\
{\tt \small imed.zaguia@rmc.ca }
}

\begin{document}

\title{Some inequalities for orderings of acyclic digraphs}
\maketitle

\begin{abstract}

Let $D=(V,A)$ be an acyclic digraph. For $x\in V$ define $e_{_{D}}(x)$ to be the difference of the indegree and the outdegree of $x$. An acyclic ordering of the vertices of $D$ is a one-to-one map
$g: V \rightarrow [1,|V|] $ that has the property that for all $x,y\in V$ if $(x,y)\in A$, then $g(x) < g(y)$.

We prove that for every acyclic ordering $g$ of $D$ the following inequality holds:
\[\sum_{x\in V} e_{_{D}}(x)\cdot g(x) ~\geq~ \frac{1}{2} \sum_{x\in V}[e_{_{D}}(x)]^2~.\]
The class of acyclic digraphs for which equality holds is determined as the class of comparbility digraphs of posets of order dimension two.
\end{abstract}

\input{epsf}

\section{An Inequality for Acyclic Digraphs}

A \emph{directed graph} (or just \emph{digraph}) $D$ consists of a
nonempty \emph{finite} set $V (D)$ of elements called \emph{vertices} and
a finite set $A(D)$ of ordered pairs of distinct vertices called
\emph{arcs}. We call $V (D)$ the \emph{vertex set} and $A(D)$ the
\emph{arc set} of $D$. We will often write $D = (V,A)$ which means
that $V$ and $A$ are the vertex set and arc set of $D$,
respectively. If $X$ is a subset of $V$, the pair
$D[X]:=(X, A\cap (X\times X))$ is the \emph{digraph
induced by $D$ on $X$}. A digraph $D$ is \emph{acyclic} if it has no directed cycle. In
this paper all digraphs are acyclic and simple in the sense that
they have no multiple arcs. For any other terminology on
digraphs we refer the reader to \cite{bang}.

A \emph{poset} $P=(V,<)$ is a set $V$ equipped with a binary relation $<$
on $V$ which is irreflexive (i.e., $x\not < x$ for all $x\in V$), antisymmetric and transitive.
To a poset $P=(V,<)$ we can associate a digraph $D(P)=(V,A)$, called the \emph{comparability digraph} of $P$, as follows. For two distinct vertices $x,y\in V$ we let $(x,y)\in A$ if $x<y$. We should mention that to an acyclic
digraph $D=(V,A)$ we can associate a poset by taking the \emph{transitive closure},
that is, the smallest binary relation on $V$ which is irreflexive, antisymmetric and transitive containing~$A$.\\

Assume that $D=(V,A)$ is an acyclic digraph. We define for $ x \in V$
\begin{equation}
N^+(x)  = \{ z \in V ~:~ (x,z) \in A \} \mbox{ and } N^-(x) = \{ z \in V : (z,x) \in A \},
\end{equation}
and let
\begin{equation} \label{E1}
 e_{_{D}}(x) = |N^-(x)|-|N^+(x)|.
\end{equation}
Every edge of a digraph goes in and comes out somewhere so we get

\begin{equation} \label{E3}
 \sum_{x\in V} e_{_{D}}(x) = \sum_{x\in V} |N^-(x)|-|N^+(x)|=0.
\end{equation}

Let $D$ be a digraph and let $x_1, x_2,\cdots, x_n$ be an ordering
of its vertices. We call this ordering an \emph{acyclic ordering}
if, for every arc $(x_i,x_j)$ in $D$, we have $i < j$. Since no directed cycle
has an acyclic ordering, no digraph with a directed cycle has an acyclic
ordering. On the other hand, every acyclic digraph has an acyclic ordering of
its vertices \cite{szp}.
Any acyclic ordering of the acyclic digraph $D=(V,A)$ defines a function
$g: V \rightarrow [1,|V|] $ by letting $g(x_i)=i$ for all $i\in [1,|V|]$. The function $g$ has the property that for all $x,y\in V$ if $(x,y)\in A$, then $g(x) < g(y)$. Conversely, any 1-1 function with this property
defines an acyclic ordering.

On the other hand we have the canonical Euclidean inner product
\begin{equation}  \label{E8}
 \langle e_{_{D}},g \rangle ~:=~ \sum_{x \in V} ~ e_{_{D}}(x) \cdot g(x) ~ \in ~ \Z ~.
\end{equation}

A \emph{linear extension} of a poset $P=(V,<)$ is an acyclic ordering of its
comparability digraph $D(P)$. The poset $P=(V,<)$
is said to have \emph{dimension two} if there are two distinct linear extensions
$f$ and $g$ such that for all $x,y\in V$, $x<y$ if and only if $f(x)<f(y)$ and $g(x)<g(y)$. In this case we write $P=f\cap g$.

Let $P=(V,<)$ be a poset of dimension two with $|V|=n$ and $D(P)$ be its comparability digraph and let $f$ and $g$ be two linear extensions of $P$ so that $P=f\cap g$. Then the following equality holds
\begin{equation}\label{E40}
e_{_{D(P)}}=f+g-(n+1).
\end{equation}
Indeed, for $x\in V$ the quantity $f(x)-(|N^-(x)|+1)$ counts the number of elements $v$ of $V$ such that $f(v)<f(x)$ and $v\not \in N^-(x)\cup\{x\}$. On the other hand the quantity
$n-(g(x)+|N^+(x)|)$ counts the number of elements $v$ of $V$ such that $g(v)>g(x)$ and $v\not \in N^+(x)\cup\{x\}$. Since $P=f\cap g$ we infer that these two quantities must be equal, that is, $e_{D(P)}(x)-f(x)-g(x)+(n+1)=0$ for all $x\in V$ as required.

A consequence of equality (\ref{E40}) is this: if $P$ has dimension at most two, then $P$ has a linear extension $g$ satisfying the equality
\begin{equation}\label{equalposet}
\langle e_{D(P)},g \rangle ~=~ \frac{1}{2}\langle e_{D(P)},e_{D(P)} \rangle.
\end{equation}
If $P$ is a total order, then consider $P$ itself as its linear extension, say $g$. If $n$ is even, then
$\langle e,g \rangle = 1\cdot (n-1)+2\cdot(n-3)+\cdot \cdot \cdot + n/2\cdot(-1)+(n/2+1)\cdot1+(n/2+2)\cdot3+\cdot \cdot \cdot +n\cdot(1-n) = (1-n)^2+(3-n)^2+(5-n)^2+\cdot \cdot \cdot +(-3)^2+(-1)^2$, and $\langle e,e \rangle= 2[(1-n)^2+(3-n)^2+(5-n)^2+\cdot \cdot \cdot + (-3)^2+(-1)^2].$
Similarly, when $n$ is odd we can obtain the equality.

If $P$ has dimension 2, then let $f$ and $g$ be linear extensions satisfying $f \cap g = P$. Then $e_{_{D(P)}}=f + g -(n + 1)$,  and since obviously $f$ and $g$ are acyclic digraphs, $\langle g,e_D \rangle$ is well-defined and in fact
\begin{eqnarray*}
\langle g,e_{_{D(P)}} \rangle &=& \langle g,f \rangle + \langle g,g \rangle -(n+1)\sum_{x\in V} g(x)\\
&=&\langle g,f \rangle + \langle g,g \rangle -\frac{n(n+1)^2}{2}.
\end{eqnarray*}
On the other hand we have
\begin{align*}
\langle e_{_{D(P)}},e_{_{D(P)}} \rangle &= \langle f + g -(n + 1),f + g -(n + 1)   \rangle\\
&= \langle f ,f \rangle + \langle f ,g \rangle -(n+1)\sum_{x\in V} f(x) + \langle g ,f \rangle + \langle g ,g \rangle -(n+1)\sum_{x\in V} g(x) \\
&\qquad {}\; \; \; \; \; \; \;\;\;\;\;\;\;\;\;\;\;\;\,-(n+1)\sum_{x\in V} f(x) -(n+1)\sum_{x\in V} g(x) +n(n+1)^2\\
&=2 \langle g ,f \rangle +2 \langle g ,g \rangle -4(n+1)\sum_{x\in V} g(x) +n(n+1)^2.
\end{align*}
Hence,
\begin{eqnarray*}
\frac{1}{2}\langle e_{_{D(P)}},e_{_{D(P)}} \rangle &=& \langle g ,f \rangle + \langle g ,g \rangle -2(n+1)\frac{n(n+1)}{2} +\frac{n(n+1)^2}{2}\\
&=& \langle g ,f \rangle + \langle g ,g \rangle -\frac{n(n+1)^2}{2}.
\end{eqnarray*}

For posets of dimension larger than 2, the first author proved in \cite{bier} that the following inequality holds
\begin{equation*}\label{inequalposet}
\langle e_{D(P)},g \rangle ~\geq~ \frac{1}{2}\langle e_{D(P)},e_{D(P)} \rangle.
\end{equation*}

On the other hand the results in \cite{bier} for the lower bound do hold for the case of acyclic digraphs also, but this needs to be checked carefully. It is the main object of this paper to provide the details of such a check.

\begin{thm} \label{T1} Let $D=(V,A)$ be an acyclic digraph. Assume
that $g:V \rightarrow [1,|V|] $ is an acyclic ordering of $D$.
Then we have the inequality
\begin{equation}  \label{E9}
 \langle e_{_{D}},g \rangle ~\geq~ \frac{1}{2}\langle e_{_{D}},e_{_{D}} \rangle~.
\end{equation}
\end{thm}

The next theorem characterizes the digraphs satisfying equality in (\ref{E9}).

\begin{thm} \label{T2} Let $D=(V,A)$ be an acyclic digraph with $n=|V|$ for which there exists an acyclic ordering $g: V\rightarrow [1,n]$ that satisfies the equality.

\begin{equation}  \label{E8'}
 \langle e_{_{D}},g \rangle ~=~ \frac{1}{2}\langle e_{_{D}},e_{_{D}} \rangle~.
\end{equation}

Then $D$ is the comparability digraph of a poset of dimension at most two, $f=n+1-g+e_{_{D}}$ is a linear extension of $D$ and $D=f\cap g$.
\end{thm}

\begin{exm}Consider the directed graph $D$ depicted in Figure \ref{skinny}.
Notice that $D$ is also a poset. The corresponding $e$-vector is $e=(-1,-2,2,1)$ and satisfies $\langle e,e \rangle=10$.
Now let $g$ be defined by $g(x_i)=i$ for all $i\in \{1,2,3,4\}$. Then
$g$ is an acyclic ordering of $D$ and $\langle e,g \rangle =5=\frac{1}{2}\langle e,e \rangle$. Moreover, $f=n+1-g+e$ is an acyclic ordering such that
$f(x_1)=3$, $f(x_2)=1$, $f(x_3)=4$ and $f(x_4)=2$. It is easily checked that $D=f\cap g$.
\end{exm}

\begin{figure}[h]
\begin{center}
\leavevmode \epsfxsize=2in \epsfbox{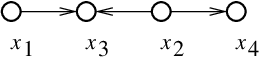}
\end{center}
\caption{Example for Theorem \ref{T2} \label{skinny}}
\end{figure}

The following result first appeared in \cite{bier}. It is now a consequence of Theorem \ref{T2} and the discussion before Theorem \ref{T1}.

\begin{cor}\label{thomas}Let $P$ be a poset. Then $P$ is of dimension at most two
if and only if $P$ has a linear extension $g$ satisfying the equality

\begin{equation*}
 \langle e_{_{D(P)}},g \rangle ~=~ \frac{1}{2}\langle e_{_{D(P)}},e_{_{D(P)}} \rangle~.
\end{equation*}

\end{cor}

Corollary \ref{thomas} gives a new characterization of posets of dimension two. We should
mention here that several other characterizations exist. In \cite{3} it was proved that a poset has dimension two if and only if
the complement of its comparability graph is a comparability graph. Other characterizations of
posets of dimension two can be found in \cite{bfr}.

\section{Proof of Theorem \ref{T1}}

We may consider for a given acyclic ordering $g$ the total sum of all its weights
induced on the arcs of $D$ and we find the expression
\begin{equation}  \label{EW}
  \sum_{(x,y)\in A}~ [g(y) - g(x)] ~=~ \sum_{x \in V} ~e_{_{D}}(x) \cdot g(x)  ~.
\end{equation}
This comes about by noticing that the left hand side sums over all arcs
and for each vertex $v \in V$ counts $+g(v)$ for each arc entering $v$ and
$-g(v)$ for each arc leaving $v$ for a total of $g(v)\cdot e(v)$. Sum over each
vertex $v$ to get the right hand side.

The proof of Theorem \ref{T1} goes by induction on the order $|V| =n$. First a lemma, already proved in \cite{bier} for posets. An element $x\in X$ is \emph{maximal} in a digraph $D=(X,A)$ if there is no $y\in V$ such that $(x,y)\in A$ i.e. $N^+(x)=\emptyset$.

\begin{lem}\label{lem1} Let $D_1=(X,A_1)$ be an acyclic digraph and let $g$ be an acyclic ordering of $D_1$.
Let $z\in X$ be a maximal element of $D_1$ and let $m=|N^-(z)|$. Then

\begin{equation} \label{E11}
 \sum_{x \in N^-(z)} [e_{_{D_{1}}}(x) - g(x)] + n \cdot m ~\geq ~  \binom m2 ~.
\end{equation}
\end{lem}

\begin{proof}For the proof of the lemma some preliminary considerations are useful: \\
For a subset $S \subseteq [1,n] $ of the integer interval and its set complement
$ T = [1,n] \setminus S $ we call any ordered pair $ (s,t) $ with $s<t$ and
$s \in S~,~t \in T$ an { \it insertion pair} of $S.$ Then if the subset
$ S = \{ s_1,s_2,...,s_m \} $ has exactly $k_S$ insertion pairs, we find that
there are exactly $(n-s_i)-(m-i)$ insertion pairs $(s_i,t)$ and hence
\begin{equation} \label{E12}
 k_S =  \sum_{i=1}^m ~ [(n-s_i) - (m-i)] ~.
\end{equation}
Thus we obtain an equality for the sum over the set $S$ as
\begin{equation} \label{E13}
 \sum_{i=1}^m ~s_i ~= ~ \frac{m(2n-m+1)}{2} ~-~k_S ~= ~ n\cdot m ~-~ \binom m2 ~-~k_S ~.
\end{equation}
This remark is now applied in the situation of the lemma. Choosing $S=\{g(x) : x\in N^-(z)\}$ from (\ref{E13}) we obtain
\begin{equation*}
\sum_{x\in N^-(z)} ~g(x) ~= ~ n\cdot m ~-~ \binom m2 ~-~k_S ~.
\end{equation*}
and hence
\begin{equation*}
n\cdot m ~+~ \sum_{x\in N^-(z)} ~[e_{_{D_{1}}}(x)-g(x)] ~= ~ \binom m2 ~+~k_S ~+~ \sum_{x\in N^-(z)} e_{_{D_1}}(x)~.
\end{equation*}

The sum $  \sum_{x \in N^-(z)} ~e_{_{D_1}}(x) $ counts the difference of the number of arcs
going into $N^-(z)$ and of the number of arcs coming out of $N^-(z).$ On the other hand $k_S$ is at least the number of arcs from $N^-(z)$ to its complement. It follows then that $k_S ~+~ \sum_{x\in N^-(z)} e_{_{D_1}}(x) ~\geq ~ 0$. The required inequality follows.
\end{proof}

We now proceed to the proof of Theorem \ref{T1}.

\begin{proof} (Of Theorem \ref{T1})
Clearly the result holds for the acyclic digraph of one element.
Assume that the result is known for all acyclic digraphs of size $n$ and that we want
to show it for an acyclic digraph with $n+1$ elements. Denote such an acyclic digraph by $D_1=(X,A_1)$ so that
$|X| = n+1 $ and let $G : X \rightarrow [1,n+1]$ be an acyclic ordering of the
acyclic digraph $D_1$. Then let $z \in X$ be the unique element with $G(z)=n+1.$
Let $V=X\setminus \{z\}$ and consider the acyclic digraph $D:=D_1[V]=(V,A)$. Clearly the restriction of $G$ to $D$ defines an acyclic ordering $g:V \rightarrow [ 1,n] $. Then we clearly have
\begin{equation} \label{E19}
  e_{D}(x) = \left\{
\begin{array}{ccc}
e_{D_1}(x) & \mbox{ if } & x \not\in N^-(z), \\
e_{D_{1}}(x)+1 & \mbox{ if } & x \in N^-(z). \end{array} \right.
\end{equation}
Note that in particular
\begin{equation} \label{E20}
 \sum_{x \in N^-(z)} ~e_{D}(x)  ~=~ \sum_{x \in N^-(z)} ~e_{D_1}(x) ~+~ |N^-(z)|.
\end{equation}

By the inductive assumption for the digraph $D$ we have:
\begin{equation} \label{E21}
  \langle g,e_D \rangle ~\geq ~ \frac{1}{2} \langle e_D,e_D\rangle.
\end{equation}
The quantity $\langle e_{D_1},e_{D_1} \rangle$ by (\ref{E19}) works out to be
\begin{eqnarray*} \label{E22}
\langle e_{D_1},e_{D_1} \rangle &=& \sum_{x\in X}{e_{D_1}(x)}^2\\
                      &=&  \sum_{x\in V}{e_{D_1}(x)}^2  + {e_{D_1}(z)}^2\\
                      &=& \sum_{x\not \in N^-(z)}{e_{D_1}(x)}^2 + \sum_{x \in N^-(z)}{e_{D_1}(x)}^2 + {|N^-(z)|}^2\\
                      &=& \sum_{x\not \in N^-(z)}{e_D(x)}^2 + \sum_{x \in N^-(z)}{(e_D(x)-1)}^2 + {|N^-(z)}|^2\\
                      &=& \langle e_D,e_D \rangle - 2 \sum_{x \in N^-(z)} ~e_D(x)  + |N^-(z)| + {|N^-(z)}|^2\\
                      &=& \langle e_D,e_D \rangle - 2  \sum_{x \in N^-(z)} ~e_D(x)  +2\,[|N^-(z)| + \binom {|N^-(z)|}2]
\end{eqnarray*}

so that
\begin{eqnarray} \label{E33}
\frac{1}{2} \langle e_D,e_D \rangle &=& \frac{1}{2}\langle e_{D_1},e_{D_1} \rangle_X + \sum_{x \in N^-(z)} ~e_D(x) - |N^-(z)| -\binom {|N^-(z)|}2
\end{eqnarray}
and where we have
\begin{eqnarray*}  \label{E23}
\langle G,e_{D_1}\rangle &=&\sum_{x\in X}G(x)\cdot e_{D_1}(x)\\
        &=& \sum_{x\in V} G(x)\cdot e_{D_1}(x) + G(z)\cdot e_{D_1}(z)\\
        &=& \sum_{x\not \in N^-(z)} g(x)\cdot e_D(x)  + \sum_{x \in N^-(z)} g(x)\cdot (e_D(x)-1)+(n+1)|N^-(z)|\\
        &=&  \langle g,e_D\rangle - \sum_{x \in N^-(z)} ~g(x)  + (n+1)\cdot |N^-(z)|
\end{eqnarray*}

so that
\begin{eqnarray} \label{E34}
\langle g,e_D\rangle &=& \langle G,e_{D_1}\rangle + \sum_{x \in N^-(z)} ~g(x)  - (n+1)\cdot |N^-(z)|
\end{eqnarray}

so that from (\ref{E21})
\begin{equation} \label{E24}
  \langle G,e_{D_1} \rangle \geq \frac{1}{2} \langle e_{D_1},e_{D_1} \rangle + \sum_{x \in N^-(z)} ~[e_D(x)-g(x)] + n\cdot |N^-(z)| - \binom {|N^-(z)|}2.
\end{equation}
We then easily see that the inequality in question (for $G$ and $e_{D_1}$) follows from Lemma 1.
\end{proof}

\section{Characterization of acyclic digraphs satisfying equality: A proof of Theorem \ref{T2}}

The proof of Theorem \ref{T2} is by induction on the order $|V| =n$. The following lemma is then essential.

\begin{lem}  \label{L3}Assume that $D_1=(X,A_1)$ is an acyclic digraph and $G : X \rightarrow [1,n+1]$ is an acyclic ordering that satisfies the equality
\begin{equation} \label{E25}
  \langle G,e_{D_1} \rangle_X ~= ~ \frac{1}{2} \langle e_{D_1},e_{D_1}\rangle.
\end{equation}
Let $z\in X$ be the unique element with $G(z)=n+1$ and let $V=X\setminus \{z\}$. Then the restriction $g:=G\restriction V$ satisfies the equality
\begin{equation} \label{E26}
  \langle g,e_D \rangle ~= ~ \frac{1}{2} \langle e_D,e_D\rangle.
\end{equation}
\end{lem}

\begin{proof}If equality (\ref{E25}) holds, then from equalities (\ref{E33}) and (\ref{E34}) we deduce that
\begin{eqnarray*}
0\leq \langle g,e_D \rangle -\frac{1}{2} \langle e_D,e_D\rangle= -[\sum_{x \in N^-(z)} ~[e_D(x)-g(x)] + n\cdot |N^-(z)| - \binom {|N^-(z)|}2]\leq 0.
\end{eqnarray*}
The first inequality follows from Theorem \ref{T1} and the last inequality follows from Lemma~\ref{lem1}.
\end{proof}

Note that under the given assumptions as in the previous lemma the set $N^-(z)$ has the property

\begin{equation}
 \sum_{x \in N^-(z)} [e_D(x) - g(x)] + n \cdot |N^-(z)| ~= ~  \binom {|N^-(z)|}{2} ~.
\end{equation}

Note that if $f:= n+1-g+e_D$, then
\begin{eqnarray*}
\binom {|N^-(z)|}{2} & = & \sum_{x \in N^-(z)} [f(x)-(n+1)] + n \cdot |N^-(z)|\\
                     & = & \sum_{x \in N^-(z)} f(x) -(n+1)\cdot |N^-(z)| + n \cdot |N^-(z)|\\
                     & = & \sum_{x \in N^-(z)} f(x) - |N^-(z)|.
\end{eqnarray*}

and hence
\begin{equation}\label{E31}
\sum_{x \in N^-(z)} f(x)=\frac{|N^-(z)|(|N^-(z)|+1)}{2}.
\end{equation}

Moreover, if $f$ is one-to-one, then the images under $f$ of the set $N^-(z)$ are the numbers in the interval $[1,|N^-(z)|]$ (this follows from Lemma \ref{sum}).

\begin{lem}\label{sum} Let $0<a_1<a_2<...<a_m$ be integers such that $\sum_{i=1}^{m}a_i=\frac{m(m+1)}{2}$. Then $a_i=i$.
\end{lem}
\begin{proof}Straightforward.
\end{proof}

We now proceed to the proof of Theorem \ref{T2}

\begin{proof} (Of Theorem \ref{T2}.) Let $D=(V,A)$ be an acyclic digraph with $n=|V|$
satisfying the conditions of the theorem. The proof is by induction on $n$. For $n=1$ all conclusions are trivially satisfied.\\
Assume as an inductive hypothesis that if the equality
$\langle g,e_D \rangle = \frac{1}{2} \langle e_D,e_D\rangle$ holds, then
$D$ is the comparability digraph of a poset of dimension two, $f=n+1-g+e_D$ is an acyclic ordering of $D$
and $D=f\cap g$. For the inductive step let $D_1=(X,A_1)$ be an acyclic digraph for which there exists an acyclic ordering
$G : X \rightarrow [1,n+1]$ that satisfies $\langle G,e_{D_1} \rangle = \frac{1}{2} \langle e_{D_1},e_{D_1}\rangle$. Let $F=n+2-G+e_{D_1}$.
Let $z$ be the unique element with $G(z)=n+1$ and set $D:=D_1[X\setminus \{z\}]$.
By Lemma \ref{L3} the restriction $g:=G_{\restriction D}$ satisfies the equality $\langle g,e_D \rangle = \frac{1}{2} \langle e_D,e_D\rangle$. Hence the inductive hypothesis applies to $D$. Note that
\begin{equation}\label{E28}
F(z)= n+2-G(z)+e_{D_1}(z)=n+2-(n+1)+|N^-(z)|=|N^-(z)|+1.
\end{equation}

We now verify that $F$ is an acyclic ordering of $D_1$ and that $D_1 =F\cap G$.

We first verify that $F$ is well defined by showing that $0<F(x)\leq n+1$ for all $x\in X$.

$F(x)=n+2-G(x)+e_{D_1}(x)=n+2-(G(x)-|N^-(x)|)-|N^+(x)|$. As $G(x)> |N^-(x)|$ it follows that $F(x)\leq n+1$.

The number $G(x)-|N^-(x)|$ counts a certain set of elements $M$ which
are outside $|N^+(x)|$ because $G$ is an acyclic ordering, and which are outside $|N^-(x)|$ because we have

\[ M\subseteq \{y\in X : G(y)\not\in \{G(t) : t\in N^-(x)\}\}. \]

As $M \cap N^+(x) =\emptyset$ we get $G(x)-|N^-(x)|+|N^+(x)|=|M|+|N^+(x)|\leq |X|=n+1$.
This verifies that $F$ is well defined.

As

\begin{equation}\label{E30}
    e_{D}(x)=\left\{ \begin{array}{ccc} e_{D_1}(x)+1 & \mbox{ if } & x\in N^-(z), \\
                                    e_{D_1}(x) & \mbox{ if } & x\not \in N^-(z),
                 \end{array}\right.
\end{equation}

we have

\begin{equation}\label{E27}
    F_{\restriction V} (x)=\left\{ \begin{array}{ccc} f(x) & \mbox{ if } & x\in N^-(z) \\
                                    f(x) +1 & \mbox{ if } & x\not \in N^-(z).
                 \end{array}\right.
\end{equation}

Since $f$ is one-to-one it follows from (\ref{E31}) that the images under $f$ of the set $N^-(z)$ are the numbers in the interval $[1,|N^-(z)|]$. Hence, the images of the complement of $N^-(z)$ in $V$ are the numbers in the interval $[1+|N^-(z)|,n]$. From (\ref{E27}) we deduce that the images under $F$ of the set $N^-(z)$ are the numbers in the interval $[1,|N^-(z)|]$, and the images of the complement of $N^-(z)$ in $V$ are the numbers in the interval $[2+|N^-(z)|,n+1]$. From (\ref{E28}) we deduce that $F$ is injective and hence bijective.

Next we verify that $F$ is an acyclic ordering. Let $(x,y)\in A$.
For the two cases where $x,y \in N^-(z)$ or $x,y \not \in N^-(z)\cup\{z\}$ the
fact that $F(x)<F(y)$ follows from (\ref{E27}) and our assumption that $f$ is an acyclic ordering of $D$.
In case $x\in N^-(z)$ and $y \not \in N^-(z)\cup \{z\}$ the
fact that $F(x)<F(y)$ follows from $F(x)=f(x)$ and $F(y)=f(y)+1$.
The case $x\not \in N^-(z)$ and $y \in N^-(z)$ cannot occur because the images under $f$ of the set $N^-(z)$ are the numbers in the interval $[1,|N^-(z)|]$.

The case $(x,z)\in A_1$ is also clear for the same reason: $F(x)=f(x)<|N^-(z)|+1=F(z)$.
This verifies that $F$ is an acyclic ordering of $D_1$.  Finally we now have to verify that
$D_1=F\cap G$. Since $D=f\cap g$ by inductive assumption, it is enough to check that if $(x,z)\not \in A_1$, then $F(x)>F(z)$ which follows from the fact that the images under $F$ of the complement of $N^-(z)$ in $X$ are the numbers in the interval $[1+|N^-(z)|,n+1]$ and $F(z)=|N^{-1}(z)|+1$. This verifies that $D_1=F\cap G$.

Hence, $D_1$ is the comparability digraph of a poset of order dimension two.
\end{proof}

\section{Acknowledgment}

The authors are grateful to an anonymous referee for the careful reading of the manuscript and for the detailed comments which have greatly improved the presentation of the paper.

\end{document}